\documentclass[12pt,leqno]{amsart}
\usepackage{mathrsfs,dsfont}
\usepackage{amsmath,amstext,amsthm,amssymb}

\newtheorem{Theorem}{Theorem}[section]
\newtheorem{Proposition}[Theorem]{Proposition}
\newtheorem{Lemma}[Theorem]{Lemma}

\theoremstyle{definition}

\DeclareMathOperator{\Span}{Span\,}

\begin{document}

\title[Geometric properties of upper level sets of Lelong numbers]{Geometric properties of upper level sets of Lelong numbers on projective spaces} 
\author{Dan Coman \and Tuyen Trung Truong}
\thanks{D. Coman is partially supported by the NSF Grant DMS-1300157}
\date{May 3, 2013}
\address{Department of Mathematics, Syracuse University, Syracuse, NY 13244-1150, USA}\email{dcoman@syr.edu}
\address{Department of Mathematics, Syracuse University, Syracuse, NY 13244-1150, USA}\email{tutruong@syr.edu}

\pagestyle{myheadings} 

\begin{abstract}
\noindent Let $T$ be a positive closed current of unit mass on the complex projective space $\mathbb P^n$. For certain values $\alpha<1$, we prove geometric properties of the set of points in $\mathbb P^n$ where the Lelong number of $T$ exceeds $\alpha$. We also consider the case of positive closed currents of bidimension (1,1) on multiprojective spaces.
\end{abstract}

\maketitle


\section{Introduction}

\par Let $T$ be a positive closed current of bidimension $(p,p)$ on a complex manifold $M$. For $\alpha\geq 0$, we consider the upper level sets of the Lelong numbers $\nu(T,z)$ of $T$,
\begin{eqnarray*}
&&E_\alpha(T)=E_\alpha(T,M)=\{z\in M:\,\nu(T,z)\geq\alpha\}\,,\\
&&E^+_\alpha(T)=E^+_\alpha(T,M)=\{z\in M:\,\nu(T,z)>\alpha\}\,.
\end{eqnarray*}
By \cite{Siu74}, if $\alpha>0$ then $E_\alpha(T)$ is an analytic subvariety of $M$ of dimension at most $p$, hence $E^+_0(T)$ is an at most countable union of analytic subvarieties of $M$ of dimension $\leq p$. 

\par If $M=\mathbb P^n$ we denote by $\|T\|$ the mass (or the degree) of $T$ computed with respect to the Fubini-Study form $\omega_n$ on $\mathbb P^n$, 
$$\|T\|=\int_{\mathbb P^n}T\wedge\omega_n^p\,.$$ 
It is well known that $\nu(T,z)\leq\|T\|$ for every $z\in\mathbb P^n$ (see e.g. \cite{CG09}). Assume without loss of generality that $\|T\|=1$. When $p=1$, it was shown in \cite[Theorem 1.1]{C06} that $E^+_{2/3}(T,\mathbb P^n)$ is contained in a (complex) line, while $E^+_{1/2}(T,\mathbb P^n)$ is either contained in a line or else it is a finite set such that $|E^+_{1/2}(T,\mathbb P^n)\setminus L|=1$ for some line $L$. In dimension two, it was proved in \cite[Theorem 1.2]{C06} that $E^+_{2/5}(T,\mathbb P^2)$ is either contained in a conic or else it is a finite set such that $|E^+_{2/5}(T,\mathbb P^2)\setminus C|=1$ for some conic $C$. When $p=n-1$, i.e. when $T$ has bidegree $(1,1)$, it was shown in \cite[Proposition 2.2]{CG09} that $E^+_{n/(n+1)}(T,\mathbb P^n)$ is contained in a hyperplane of $\mathbb P^n$. Moreover, these values of $\alpha$ are sharp with respect to these geometric properties. 

\par In this paper we generalize these results to the case of currents of arbitrary bidimension on $\mathbb P^n$. Namely, we prove the following theorems.

\begin{Theorem}\label{T:mt1} If $T$ is a positive closed current of bidimension $(p,p)$ on $\mathbb P^n$, $0<p<n$, with $\|T\|=1$, then the set $E^+_{(p+1)/(p+2)}(T,\mathbb P^n)$ is contained in a $p$-dimensional linear subspace of $\mathbb P^n$.
\end{Theorem}

\par Decreasing the value of $\alpha$ to $p/(p+1)$ we show that $E^+_{p/(p+1)}(T,\mathbb P^n)$ has all but at most one point contained in a $p$-dimensional linear subspace of $\mathbb P^n$. More precisely, the following holds:

\begin{Theorem}\label{T:mt2} If $T$ is a positive closed current of bidimension $(p,p)$ on $\mathbb P^n$, $0<p<n$, with $\|T\|=1$, then the set $E^+_{p/(p+1)}(T,\mathbb P^n)$ is either contained in a $p$-dimensional linear subspace of $\mathbb P^n$ or else it is a finite set and $|E^+_{p/(p+1)}(T,\mathbb P^n)\setminus L|=p$ for some line $L$. 
\end{Theorem}

\par Theorems \ref{T:mt1} and \ref{T:mt2} are proved in Section \ref{S:mtpf}. We also prove there another property of the set $E^+_{p/(p+1)}(T,\mathbb P^n)$ (see Proposition \ref{P:charge}), and we give examples of currents showing that the values of $\alpha$ in Theorems \ref{T:mt1} and \ref{T:mt2} are sharp with respect to the corresponding geometric property. 

\par Decreasing the value of $\alpha$ further to $(3p-1)/(3p+2)$, we obtain a result analogous to Theorem 1.2 in \cite{C06}:

\begin{Theorem}\label{T:mt3} Let $T$ be a positive closed current of bidimension $(p,p)$ on $\mathbb P^n$ such that $1<p<n$, $\|T\|=1$, and the set $E^+_{(3p-1)/(3p+2)}(T,\mathbb P^n)$ is not contained in a $p$-dimensional linear subspace of $\mathbb P^n$. If $W=\Span(E^+_{(3p-1)/(3p+2)}(T,\mathbb P^n))$ then $\dim W=p+1$ and there exist plane conics $C_j\subset W$ and points $z_j\in W$, $1\leq j\leq N_p$, where $N_p=\binom{p+2}{3}$, such that $z_j$ lies in the plane containing $C_j$ and 
$$E^+_{(3p-1)/(3p+2)}(T,\mathbb P^n)\subset C_1\cup\ldots\cup C_{N_p}\cup\{z_1,\ldots,z_{N_p}\}\,.$$
\end{Theorem}

\par Theorem \ref{T:mt3} is proved in Section \ref{S:mt3pf}. We note there that the corresponding statement does not hold for currents of bidimension (1,1). However, we give in Theorem \ref{T:mt31} a geometric description of the set $E^+_{2/5}(T,\mathbb P^n)$ when $T$ is a positive closed current of bidimension $(1,1)$ on $\mathbb P^n$ such that $\|T\|=1$. The proof of Theorem \ref{T:mt31} uses ideas from \cite{D92,Me98,Vig} related to self-intersection inequalities for positive closed currents. Using similar ideas, we further show in Theorem \ref{T:KC} that if $T$ is a positive closed current of bidimension $(p,p)$ on a compact K\"ahler manifold $(X,\omega)$, then the set $E_c(T,X)$, $c>0$, is contained in an analytic set of dimension $\leq p$ whose volume and number of irreducible components are bounded above by a positive constant which depends only on $\|T\|$ and $c$. In this case, the volumes of analytic subvarieties of $X$ and the mass $\|T\|$ are computed with respect to the fixed K\"ahler form $\omega$ on $X$. 

\par In Section \ref{S:multiproj} we study similar geometric properties for upper level sets of positive closed currents of bidimension $(1,1)$ on multi-projective spaces.

\section{Proofs of Theorems \ref{T:mt1} and \ref{T:mt2}}\label{S:mtpf} 

\subsection{Proof of Theorem \ref{T:mt1}} Let us start by recalling some terminology. If $A\subset\mathbb P^n$ we denote by $\Span(A)$ the smallest linear subspace of $\mathbb P^n$ containing $A$. We say that the points $x_1,\dots,x_{k+1}\in\mathbb P^n$, $k\leq n$, are linearly independent if they span a $k$-dimensional linear subspace of $\mathbb P^n$. We say that $k>n+1$ points of $\mathbb P^n$ are in general position if any $n+1$ of them are linearly independent. We will need the following lemma:

\begin{Lemma}\label{L:Meo} Let $T$ be a positive closed current of bidimension $(p,p)$ on $\mathbb P^n$, $0<p<n$, let $\alpha>0$, and let $V$ be a $(p+1)$-dimensional linear subspace of $\mathbb P^n$. There exists a positive closed current $S$ of bidegree $(1,1)$ on $V\equiv\mathbb P^{p+1}$ such that $\|S\|=\|T\|$ and $E^+_\alpha(T,\mathbb P^n)\cap V\subset E^+_\alpha(S,V)$.
\end{Lemma}

\begin{proof} Assume without loss of generality that $\|T\|=1$. By \cite{Me98}, there exists a positive closed current $T'$ of bidegree $(1,1)$ on $\mathbb P^n$ such that $\|T'\|=1$ and $\nu(T',z)=\nu(T,z)$ for every $z\in\mathbb P^n$. Demailly's regularization theorem \cite[Proposition 3.7]{D92} yields a sequence of positive closed currents $T'_m$ of bidegree $(1,1)$ on $\mathbb P^n$ with $\|T'_m\|=1$, such that each $T'_m$ is smooth outside an analytic subset contained in $E^+_0(T')$ and $\lim_{m\to\infty}\nu(T'_m,x)=\nu(T',x)$ at each $x\in{\mathbb P}^n$. We write $T'_m=\omega_n+dd^c\varphi_m$ for some $\omega_n$-plurisubharmonic function $\varphi_m$ on $\mathbb P^n$ (see e.g. \cite{GZ05}). Note that by \cite{Siu74}, $E^+_0(T')=E^+_0(T)$ is a countable union of analytic subsets of dimension at most $p$, so $V\setminus E^+_0(T')\neq\emptyset$. Since $\varphi_m$ is smooth at each point of $V\setminus E^+_0(T')$ the pull-back $S_m$ of $T'_m$ to $V$, $S_m=\omega_n\vert_{_V}+dd^c(\varphi_m\vert_{_V})$, is a well defined positive closed current of bidegree $(1,1)$ on $V$ with $\|S_m\|=1$. By passing to a subsequence, we may assume that $S_m$ converges weakly to a positive closed current $S$ of bidegree $(1,1)$ on $V$. Then $\|S\|=1$ and 
$$\nu(S,z)\geq\limsup_{m\to\infty}\nu(S_m,z)\geq\lim_{m\to\infty}\nu(T'_m,z)=\nu(T',z)\,,$$
for all $z\in V$. It follows that $E^+_\alpha(T,\mathbb P^n)\cap V\subset E^+_\alpha(S,V)$.
\end{proof}

\begin{proof}[Proof of Theorem \ref{T:mt1}] Assume for a contradiction that $\dim\Span(E^+_{(p+1)/(p+2)}(T,\mathbb P^n))\geq p+1$, so there exist linearly independent points $x_1,\dots,x_{p+2}\in E^+_{(p+1)/(p+2)}(T,\mathbb P^n)$. Let $V\equiv\mathbb P^{p+1}$ be the linear subspace spanned by these points. By Lemma \ref{L:Meo}, there exists a positive closed current $S$ of bidegree $(1,1)$ on $V$ such that $\|S\|=1$ and $x_1,\dots,x_{p+2}\in E^+_{(p+1)/(p+2)}(S,V)$. This is in contradiction to \cite[Proposition 2.2]{CG09}, which shows that $E^+_{(p+1)/(p+2)}(S,V)$ must be contained in a hyperplane of $V$. 
\end{proof}

\subsection{Proof of Theorem \ref{T:mt2}} 

\par We prove first Theorem \ref{T:mt2} for currents of bidegree $(1,1$) on $\mathbb P^n$. This is the contents of Theorem \ref{T:mt4}. Let $\alpha_n=(n-1)/n$. We begin with the following lemma:

\begin{Lemma}\label{L:redux} Let $T$ be a positive closed current of bidegree $(1,1)$ on $\mathbb P^n$ such that $\|T\|=1$ and $E^+_{\alpha_n}(T,\mathbb P^n)$ contains a set $A=\{x_1,\dots,x_{n+1}\}$ of linearly independent points. Then:

\smallskip

(i) For every subset $B\subset A$ with $|B|=k+1$, $k\geq1$, there exists a positive closed current $R_B$ of bidegree $(1,1)$ on $\Span(B)\equiv\mathbb P^k$ such that $\|R_B\|=1$ and $E^+_{\alpha_n}(T,\mathbb P^n)\cap\Span(B)\subset E^+_{\alpha_k}(R_B,\Span(B))$.

\smallskip

(ii) $E^+_{\alpha_n}(T,\mathbb P^n)\subset\bigcup_{1\leq j<k\leq n+1}L_{jk}$, where $L_{jk}$ is the line spanned by $x_j$ and $x_k$. 
\end{Lemma}

\begin{proof} $(i)$ It suffices to prove $(i)$ for $k=n-1$. Then we apply this inductively to obtain the result for arbitrary $k$. Assume without loss of generality that $B=\{x_1,\dots,x_n\}$ and let $H$ be the hyperplane spanned by $B$. Siu's decomposition theorem \cite{Siu74} implies that $T=a[H]+R$, where $[H]$ denotes the current of integration along $H$, $0\leq a\leq 1$, and $R$ is a positive closed current of bidegree $(1,1)$ on $\mathbb P^n$ with generic Lelong number 0 along $H$. We have 
$$1-a=\|R\|\geq\nu(R,x_{n+1})=\nu(T,x_{n+1})>\alpha_n\,,\;\text{ so }\,a<1-\alpha_n\,.$$
The current $R'=R/(1-a)$ has mass $\|R'\|=1$, generic Lelong number 0 along $H$, and if $z\in E^+_{\alpha_n}(T,\mathbb P^n)\cap H$ then, since $a<1-\alpha_n$, 
\begin{equation}\label{e:rec}
\nu(R',z)=\frac{\nu(T,z)-a}{1-a}>\frac{\alpha_n-a}{1-a}>\frac{2\alpha_n-1}{\alpha_n}=\alpha_{n-1}\,.
\end{equation}
By \cite[Proposition 3.7]{D92} there exists a sequence of positive closed currents $R'_m$ of bidegree $(1,1)$ on $\mathbb P^n$ with analytic singularities, such that $\|R'_m\|=1$, $\nu(R'_m,x)\leq\nu(R',x)$ and $\lim_{m\to\infty}\nu(R'_m,x)=\nu(R',x)$ for all $x\in\mathbb P^n$. It follows that $R'_m$ is smooth at each point of $H$ outside an analytic subset of $H$, so the pull-back $R'_m\vert_{_H}$ of $R'_m$ to $H$ is well-defined. Arguing as in the proof of Lemma \ref{L:Meo} we obtain the current $R_B$ that verifies the desired properties as a weak limit point of $\{R'_m\vert_{_H}\}$.

\par $(ii)$ Let $H_j$ denote the hyperplane spanned by $A\setminus\{x_j\}$. We show first that 
\begin{equation}\label{e:redux}
E^+_{\alpha_n}(T,\mathbb P^n)\subset\bigcup_{j=1}^{n+1}H_j\,.
\end{equation}
Assume that there exists $x_{n+2}\in E^+_{\alpha_n}(T,\mathbb P^n)\setminus\bigcup_{j=1}^{n+1}H_j$ and choose $x_0\in\mathbb P^n\setminus\{x_1,\dots,x_{n+2}\}$ so that the points $x_0,\dots,x_{n+2}$ are in general position and $\nu(T,x_0)=0$. By \cite[Proposition 3.7]{D92} there exists a a positive closed current $T'$ of bidegree $(1,1)$ on $\mathbb P^n$ with analytic singularities, such that $\|T'\|=1$, $\nu(T',x_j)>\alpha_n$, $j=1,\ldots,n+2$, and $T'$ is smooth near $x_0$. Let $C$ be the unique rational normal curve passing through the points $x_0,\dots,x_{n+2}$ (see \cite[p. 530]{GH}). It follows by \cite{D93} and \cite{FS95} that the measure $T'\wedge[C]$ is well defined, where $[C]$ denotes the current of integration along $C$. Since $C$ has degree $n$ and using \cite[Corollary 5.10]{D93}, we obtain
\begin{equation}\label{e:rnc}
n=\int_{\mathbb P^n}T'\wedge[C]\geq\sum_{j=1}^{n+2}T'\wedge[C](\{x_j\})\geq\sum_{j=1}^{n+2}\nu(T',x_j)\,\nu([C],x_j)>(n+2)\,\alpha_n\;,
\end{equation}
a contradiction. This proves \eqref{e:redux}. 

\par Let now $B_j=A\setminus\{x_j\}$. Applying \eqref{e:redux} to the current $R_{B_j}$ given by $(i)$ we obtain 
$$E^+_{\alpha_n}(T,\mathbb P^n)\cap H_j\subset E^+_{\alpha_{n-1}}(R_{B_j},H_j)\subset \bigcup_{k=1,k\neq j}^{n+1}\Span(A\setminus\{x_j,x_k\})\,.$$
Together with \eqref{e:redux} this implies that 
$$E^+_{\alpha_n}(T,\mathbb P^n)\subset \bigcup_{1\leq j<k\leq n+1}\Span(A\setminus\{x_j,x_k\})\,.$$
Repeating this argument inductively yields $(ii)$.
\end{proof}

\begin{Theorem}\label{T:mt4} If $T$ is a positive closed current of bidegree $(1,1)$ on $\mathbb P^n$ with $\|T\|=1$, then the set $E^+_{\alpha_n}(T,\mathbb P^n)$ is either contained a hyperplane or else it is a finite set and $|E^+_{\alpha_n}(T,\mathbb P^n)\setminus L|=n-1$ for some line $L\subset\mathbb P^n$. 
\end{Theorem}

\begin{proof} If $E^+_{\alpha_n}(T,\mathbb P^n)$ is not contained in a hyperplane then $\dim\Span(E^+_{\alpha_n}(T,\mathbb P^n))=n$ and there exists a set of linearly independent points $A=\{x_1,\dots,x_{n+1}\}\subset E^+_{\alpha_n}(T,\mathbb P^n)$. Let $L_{jk}$ denote the line spanned by $x_j,x_k$. 

\par If $E^+_{\alpha_n}(T,\mathbb P^n)=A$ then $|E^+_{\alpha_n}(T,\mathbb P^n)\setminus L_{12}|=n-1$ and we are done. Suppose that there exists $x\in E^+_{\alpha_n}(T,\mathbb P^n)\setminus A$. By Lemma \ref{L:redux} we have, after relabeling points if necessary, that $x\in L_{12}$. 

\par We show that $E^+_{\alpha_n}(T,\mathbb P^n)\subset A\cup L_{12}$. Assume for a contradiction that there exists $y\in E^+_{\alpha_n}(T,\mathbb P^n)\setminus(A\cup L_{12})$. Then $y\in L_{jk}$ for some $3\leq j<k\leq n+1$. Indeed, if $y\in L_{1k}$ or if $y\in L_{2k}$, $k\geq 3$, then let $B=\{1,2,k\}$ and $R_B$ be the current on $\Span(B)\equiv\mathbb P^2$ provided by Lemma \ref{L:redux}. Then $\{x,y,x_1,x_2,x_k\}\subset E^+_{1/2}(R_B,\mathbb P^2)$, and the set $\{x,y,x_1,x_2,x_k\}$ has at least two points outside each complex line. This is in contradiction to \cite[Theorem 1.1]{C06}. Hence after relabeling points if necessary we have that $y\in L_{34}$. 

\par Consider now the set $B=\{x_1,x_2,x_3,x_4\}$ and the current $R=R_B$ on $\Span(B)\equiv\mathbb P^3$ given by Lemma \ref{L:redux}, so $\{x,y,x_1,x_2,x_3,x_4\}\subset E^+_{2/3}(R,\mathbb P^3)$. If $V_1=\Span(\{x_2,x_3,x_4\})$, $V_3=\Span(\{x_1,x_2,x_4\})$, we write, using \cite{Siu74}, $R=a[V_1]+b[V_3]+R'$, where $R'$ has generic Lelong number 0 on $V_1\cup V_3$, $\|R'\|=1-a-b$, and 
\begin{eqnarray*}
&&\nu(R',x_1)>\frac{2}{3}-b\,,\;\nu(R',x_3)>\frac{2}{3}-a\,,\;\nu(R',x_j)>\frac{2}{3}-a-b\,,\;j=2,4\,,\\
&&\nu(R',x)>\frac{2}{3}-b\,,\;\nu(R',y)>\frac{2}{3}-a\,.
\end{eqnarray*}
Note that $a+b<1$. By \cite[Proposition 3.7]{D92} there exists a positive closed current with analytic singularities $S$ of bidegree $(1,1)$ on $\mathbb P^3$ with $\|S\|=1-a-b$ and such that the Lelong numbers of $S$ satisfy the same inequalities as those of $R'$ at the points $x,y,x_1,x_2,x_3,x_4$. Moreover, $S$ is smooth at each point where $R'$ has 0 Lelong number. Let $C_1$ be an irreducible conic in $V_1$ passing through $x_2,x_3,y$ and a point $w_1\in V_1$ where $\nu(R',w_1)=0$. Let $C_3$ be an irreducible conic in $V_3$ passing through $x_1,x_4,x$ and a point $w_3\in V_3$ where $\nu(R',w_3)=0$. Then the measures $S\wedge[C_j]$, $j=1,3$, are well defined and 
$$4(1-a-b)=\int_{\mathbb P^3}S\wedge([C_1]+[C_3])\geq\nu(S,x)+\nu(S,y)+\sum_{j=1}^4\nu(S,x_j)>4-4a-4b\,,$$
a contradiction. 

\par We conclude that $E^+_{\alpha_n}(T,\mathbb P^n)\subset A\cup L_{12}$, hence $|E^+_{\alpha_n}(T,\mathbb P^n)\setminus L_{12}|=n-1$. If $B=\{x_1,x_2,x_3\}$ and $R_B$ is the current on $\Span(B)\equiv\mathbb P^2$ given by Lemma \ref{L:redux} then $E^+_{\alpha_n}(T,\mathbb P^n)\cap L_{12}\subset E^+_{1/2}(R_B,\mathbb P^2)$. By \cite[Theorem 1.1]{C06}, the set $E^+_{1/2}(R_B,\mathbb P^2)$ is finite since it is not contained in a complex line. It follows that $E^+_{\alpha_n}(T,\mathbb P^n)$ is a finite set.
\end{proof}

\medskip

\par Theorem \ref{T:mt2} for arbitrary $p$ follows at once from Theorem \ref{T:mt4} and the next proposition.

\begin{Proposition}\label{P:redux} Let $T$ be a positive closed current of bidimension $(p,p)$ on $\mathbb P^n$, $0<p<n-1$, with $\|T\|=1$ and such that $E^+_{p/(p+1)}(T,\mathbb P^n)$ is not contained in a $p$-dimensional linear subspace of $\mathbb P^n$. If $W=\Span(E^+_{p/(p+1)}(T,\mathbb P^n))$ then $\dim W=p+1$ and there exists a positive closed current $R$ of bidegree $(1,1)$ on $W\equiv\mathbb P^{p+1}$ such that $\|R\|=1$ and $E^+_{p/(p+1)}(T,\mathbb P^n)\subset E^+_{p/(p+1)}(R,W)$.
\end{Proposition}

\begin{proof} By hypothesis $\dim W\geq p+1$. Assume for a contradiction that there exist linearly independent points $x_1,\dots,x_{p+3}\in E^+_{p/(p+1)}(T,\mathbb P^n)$. Let $U=\Span(\{x_1,\dots,x_{p+2}\})$ and pick $y\in U$ so that the points $x_1,\dots,x_{p+2},y$ are in general position in $U\equiv\mathbb P^{p+1}$. We will construct a positive closed current $S$ of bidegree $(1,1)$ on $U$ such that $\|S\|=1$ and $\{x_1,\dots,x_{p+2},y\}\subset E^+_{p/(p+1)}(S,U)$. By Lemma \ref{L:redux} $(ii)$, $y$ must lie in a line spanned by some $x_j,x_k$, $1\leq j<k\leq p+2$. This contradicts the fact that the points $x_1,\dots,x_{p+2},y$ are in general position in $U$. The construction of $S$ is as follows. Choose a sequence of points $y_m\in W\setminus U$ such that $y_m\to y$. Then the points $x_1,\dots,x_{p+2},y_m$ are linearly independent. Let $F_m$ be an automorphism of $\mathbb P^n$ such that $F_m(x_j)=x_j$, $1\leq j\leq p+2$, $F_m(x_{p+3})=y_m$ and set $T_m=(F_m)_\star T$. These are positive closed currents of bidimension $(p,p)$ on $\mathbb P^n$ with $\|T_m\|=1$ and $\nu(T_m,x_j)=\nu(T,x_j)$, $1\leq j\leq p+2$, $\nu(T_m,y_m)=\nu(T,x_{p+3})$. By passing to a subsequence we may assume that $T_m$ coverge weakly to a current $T'$. Then $\|T'\|=1$ and by \cite{D93}, 
\begin{eqnarray*}
\nu(T',x_j)&\geq&\limsup_{m\to\infty}\nu(T_m,x_j)=\nu(T,x_j)>\frac{p}{p+1}\,,\;1\leq j\leq p+2\,,\\
\nu(T',y)&\geq&\limsup_{m\to\infty}\nu(T_m,y_m)=\nu(T,x_{p+3})>\frac{p}{p+1}\,.
\end{eqnarray*}
Now Lemma \ref{L:Meo} applied to $T'$ and $U$ with $\alpha=p/(p+1)$ yields the desired current $S$. 

\medskip

\par Hence we have shown that $\dim W=p+1$. Lemma \ref{L:Meo} yields a positive closed current $R$ of bidegree $(1,1)$ on $W\equiv\mathbb P^{p+1}$ such that $\|R\|=1$ and $E^+_{p/(p+1)}(T,\mathbb P^n)=E^+_{p/(p+1)}(T,\mathbb P^n)\cap W\subset E^+_{p/(p+1)}(R,W)$ and the proposition is proved.
\end{proof}

\subsection{Remarks} We start with some examples showing that Theorems \ref{T:mt1} and \ref{T:mt2} are sharp. Let $0<p<n$ and $A=\{x_1,\dots,x_{p+2}\}$ be a set of linearly independent points of $\mathbb P^n$. We set $V_j=\Span(A\setminus\{x_j\})$ and denote by $L_{jk}$ the line spanned by $x_j,x_k$. 

\par Let $T_1=\frac{1}{p+2}\,\sum_{j=1}^{p+2}\,[V_j]\,$. Then $\|T_1\|=1$ and $E_{(p+1)/(p+2)}(T_1,\mathbb P^n)=A$ is not contained in a $p$-dimensional linear subspace of $\mathbb P^n$, so the value $\alpha=(p+1)/(p+2)$ in Theorem \ref{T:mt1} is sharp. 

\par If $p=1$ the value $\alpha=1/2$ in Theorem \ref{T:mt2} was shown to be sharp in \cite{C06}. Assume that $2\leq p\leq n-1$, choose points $x\in L_{12}\setminus A$, $y\in L_{34}\setminus A$, and let $V_x$, resp. $V_y$, denote the $p$-dimensional linear subspace of $\mathbb P^n$ spanned by $(A\cup\{x\})\setminus\{x_1,x_2\}$, resp. by $(A\cup\{y\})\setminus\{x_3,x_4\}$. Note that 
\begin{eqnarray*}
&&\{x,y\}\subset V_x\cap V_y\cap V_j \text{ for } j\geq 5\,,\;\{x_1,x_2\}\subset V_y\setminus V_x\,,\;\{x_3,x_4\}\subset V_x\setminus V_y\,,\\
&&x\in(V_3\cap V_4)\setminus(V_1\cup V_2)\,,\;y\in(V_1\cap V_2)\setminus(V_3\cup V_4)\,.
\end{eqnarray*}
It follows that the bidimension $(p,p)$ current 
$$T_2=\frac{1}{2(p+1)}\,\left(\sum_{j=1}^4\,[V_j]+[V_x]+[V_y]\right)+\frac{1}{p+1}\,\sum_{j=5}^{p+2}\,[V_j]$$
has mass $\|T_2\|=1$ and $\nu(T_2,x)=\nu(T_2,y)=\nu(T_2,x_j)=p/(p+1)$, $1\leq j\leq p+2$. Thus $E_{p/(p+1)}(T_2,\mathbb P^n)\supset A\cup\{x,y\}$, so $|E_{p/(p+1)}(T_2,\mathbb P^n)\setminus L|\geq p+1$ for every line $L\subset\mathbb P^n$. Hence the value $\alpha=p/(p+1)$ in Theorem \ref{T:mt2} is sharp.

\par One can construct a positive closed current $T_3$ of bidimension $(p,p)$ on $\mathbb P^n$ with $\|T_3\|=1$, for which 
$E^+_{p/(p+1)}(T_3,\mathbb P^n)$ is a countable union of linear subspaces of dimension at most $p-1$ contained in a $p$-dimensional linear subspace of $\mathbb P^n$. Indeed, let $V,\,V_j$, $j\geq1$, be distinct $p$-dimensional linear subspaces of $\mathbb P^n$ such that $V\cap V_j\neq\emptyset$ for all $j$, and set 
$$T_3=\frac{p}{p+1}\,[V]+\frac{1}{p+1}\,\sum_{j=1}^\infty\,2^{-j}\,[V_j]\,.$$ 

\par Finally, given any $k\geq2$, one can construct a positive closed current $T_4$ of bidimension $(p,p)$ on $\mathbb P^n$ such that $\|T_4\|=1$, $|E^+_{p/(p+1)}(T_4,\mathbb P^n)\setminus L|=p$ and $|E^+_{p/(p+1)}(T_4,\mathbb P^n)\cap L|=k$, for some line $L$. Indeed, pick distinct points $y_j\in L_{12}\setminus\{x_1,x_2\}$, $1\leq j\leq k-2$, and let $W_j=\Span(\{y_j,x_3,\dots,x_{p+2}\})$. If $0<\varepsilon<\frac{1}{k-1}$ let 
$$T_4=\frac{p-\varepsilon}{p(p+1)}\,\sum_{j=3}^{p+2}\,[V_j]+\frac{1+\varepsilon}{k(p+1)}\left([V_1]+[V_2]+\sum_{j=1}^{k-2}\,[W_j]\right)\,.$$
Then $\|T_4\|=1$ and 
$$\nu(T_4,x_1)=\nu(T_4,x_2)=\nu(T_4,y_j)=\frac{p-\varepsilon}{p+1}+\frac{1+\varepsilon}{k(p+1)}=\frac{p}{p+1}+\frac{1-(k-1)\varepsilon}{k(p+1)}>\frac{p}{p+1}\,,$$
$$\nu(T_4,x_j)=\frac{(p-1)(p-\varepsilon)}{p(p+1)}+\frac{1+\varepsilon}{p+1}=\frac{p}{p+1}+\frac{\varepsilon}{p(p+1)}>\frac{p}{p+1}\,,\;j\geq3\,.$$
Hence $T_4$ satisfies the desired properties with $L=L_{12}$. 

\medskip

\par We conclude this section by showing the following property of the set $E^+_{p/(p+1)}(T,\mathbb P^n)$:

\begin{Proposition}\label{P:charge}
Let $T$ be a positive closed current of bidimension $(p,p)$ on $\mathbb P^n$, $0<p<n$, with $\|T\|=1$, such that the set $E^+_{p/(p+1)}(T,\mathbb P^n)$ contains the linearly independent points $x_1,\dots,x_{p+1}$. If $V=\Span(\{x_1,\dots,x_{p+1}\})$ and $c$ is the generic Lelong number of $T$ along $V$ then $c>0$. 
\end{Proposition}

\begin{proof} Assume that $c=0$. Applying \cite{Me98} and \cite[Proposition 3.7]{D92} as in the proof of Lemma \ref{L:Meo} we obtain a positive closed current $S$ of bidegree $(1,1)$ on $\mathbb P^n$, with analytic singularities, such that $\|S\|=1$, $\nu(S,x_j)>p/(p+1)$ for $1\leq j\leq p+1$, and $S$ is smooth at each point of $V\equiv\mathbb P^p$ outside an analytic subset of $V$. Then the pull-back $R$ of $S$ to $V$ is well defined, it has unit mass and Lelong number $>p/(p+1)$ at the linearly independent points $x_j$. This contradicts Theorem \ref{T:mt1} (or \cite[Proposition 2.2]{CG09}).
\end{proof}

\section{Proof of Theorem \ref{T:mt3}}\label{S:mt3pf} 

\par We prove first Theorem \ref{T:mt3} for currents of bidegree $(1,1)$ on $\mathbb P^n$, $n\geq3$. This is done in the following lemma. Let $\beta_n=(3n-4)/(3n-1)$. 

\begin{Lemma}\label{L:redux2} Let $T$ be a positive closed current of bidegree $(1,1)$ on $\mathbb P^n$, $n\geq3$, such that $\|T\|=1$ and $E^+_{\beta_n}(T,\mathbb P^n)$ contains a set $A=\{x_1,\dots,x_{n+1}\}$ of linearly independent points. Then:

\smallskip

(i) For every subset $B\subset A$ with $|B|=k+1$, $k\geq2$, there exists a positive closed current $R_B$ of bidegree $(1,1)$ on $\Span(B)\equiv\mathbb P^k$ such that $\|R_B\|=1$ and $E^+_{\beta_n}(T,\mathbb P^n)\cap\Span(B)\subset E^+_{\beta_k}(R_B,\Span(B))$.

\smallskip

(ii) $E^+_{\beta_n}(T,\mathbb P^n)\subset\bigcup_{1\leq j<k<l\leq n+1}P_{jkl}$, where $P_{jkl}=\Span(\{x_j,x_k,x_l\})$. 

\smallskip

(iii) There exist conics $C_{jkl}\subset P_{jkl}$ and points $z_{jkl}\in P_{jkl}$ such that $E^+_{\beta_n}(T,\mathbb P^n)\subset\bigcup_{1\leq j<k<l\leq n+1}\big(C_{jkl}\cup\{z_{jkl}\}\big)$.
\end{Lemma}

\begin{proof} Assertions $(i)$ and $(ii)$ are shown exactly as in the proof of Lemma \ref{L:redux}, using in \eqref{e:rec} the fact that $(2\beta_n-1)/\beta_n=\beta_{n-1}$, and in \eqref{e:rnc} the fact that $n>(n+2)\,\beta_n$ implies $n\leq2$, which contradicts the assumption that $n\geq3$. 

\par $(iii)$ Let $B=\{x_j,x_k,x_l\}$, where $1\leq j<k<l\leq n+1$. By $(i)$ there exists a positive closed current $R$ of bidegree $(1,1)$ on $P_{jkl}$ such that $\|R\|=1$ and $E^+_{\beta_n}(T,\mathbb P^n)\cap P_{jkl}\subset E^+_{2/5}(R,P_{jkl})$. Theorem 1.2 in \cite{C06} shows that there exist a conic $C_{jkl}\subset P_{jkl}$ and a point $z_{jkl}\in P_{jkl}$ such that $E^+_{2/5}(R,P_{jkl})\subset C_{jkl}\cup\{z_{jkl}\}$. Hence $(iii)$ follows from $(ii)$. 
\end{proof}

\par The next proposition is proved exactly like Proposition \ref{P:redux}, by using Lemma \ref{L:redux2} $(ii)$. 

\begin{Proposition}\label{P:redux2} Let $T$ be a positive closed current of bidimension $(p,p)$ on $\mathbb P^n$ such that $1<p<n-1$, $\|T\|=1$, and the set $E^+_{(3p-1)/(3p+2)}(T,\mathbb P^n)$ is not contained in a $p$-dimensional linear subspace of $\mathbb P^n$. If $W=\Span(E^+_{(3p-1)/(3p+2)}(T,\mathbb P^n))$ then $\dim W=p+1$ and there exists a positive closed current $R$ of bidegree $(1,1)$ on $W\equiv\mathbb P^{p+1}$ such that $\|R\|=1$ and $E^+_{(3p-1)/(3p+2)}(T,\mathbb P^n)\subset E^+_{(3p-1)/(3p+2)}(R,W)$.
\end{Proposition}

\par Theorem \ref{T:mt3} for arbitrary $p\geq2$ now follows at once from Proposition \ref{P:redux2} and from Lemma \ref{L:redux2} $(iii)$.

\bigskip

\par We now turn our attention to the case of currents of bidimension $(1,1)$. If $L_1,L_2$ are non-concurrent lines in $\mathbb P^n$ and $T=([L_1]+[L_2])/2$ then $\dim\Span(E^+_{2/5}(T))=3$, and Theorem \ref{T:mt3} does not hold for $p=1$. However, we have the following geometric property of the set $E^+_{2/5}(T)$ in this setting:

\begin{Theorem}\label{T:mt31}
Let $T$ be a positive closed current of bidimension $(1,1)$ on $\mathbb{P}^n$ with $\|T\|=1$. If $|E^+_{2/5}(T)|>37$ then there exists a curve $C\subset\mathbb P^n$ of degree at most $2$ such that $| E^+_{2/5}(T)\setminus C |\leq1$. 
\end{Theorem}

\begin{proof} We consider two mutually exclusive cases. 

\smallskip

\par Case 1: The set $E^+_\gamma(T)$ is infinite for some $\gamma>1/3$. Then, by \cite{Siu74}, $E_\gamma(T)$ contains an irreducible curve $X$ and $T=T'+\gamma[X]$, where $T'$ is a positive closed current and $\deg X\leq1/\gamma<3$. If $\deg X=2$ then, by \cite[Proposition 0]{EH87}, $X$ is an irreducible plane conic. Moreover, $\|T'\|=1-2\gamma<1/3$, so $E^+_{2/5}(T)\subset X$. If $\deg X=1$ then $\|T'\|=1-\gamma<2/3$. It follows by \cite[Theorem 1.1]{C06} that $|E_{1/3}(T')\setminus L|\leq1$ for some line $L$. Since $E^+_{2/5}(T)\subset X\cup E_{1/3}(T')$ we conclude that $|E^+_{2/5}(T)\setminus C |\leq1$, where $C=X\cup L$. 

\smallskip

\par Case 2: The set $E^+_\gamma(T)$ is finite for all $\gamma>1/3$. By \cite{Me98}, there is a positive closed current $S$ of bidegree $(1,1)$ on $\mathbb P^n$ such that $||S||=||T||= 1$, and $S$ has the same Lelong number as $T$ at every point. Fix $\gamma\in(1/3,2/5)$. By Demailly's regularization theorem applied to $S$ (Main Theorem 1.1 in \cite{D92}, where we can take $u=0$ since we work on $\mathbb P^n$), for any $\epsilon >0$ there is a positive closed current $S_{\epsilon ,\gamma }$ of bidegree $(1,1)$ on $\mathbb P^n$ with the following properties:

\par (i) $S_{\epsilon ,\gamma}$ is smooth on $\mathbb{P}^n\setminus E_\gamma(T)$, hence $S_{\epsilon ,\gamma}$ is smooth outside a finite set. 

\par (ii) $\|S_{\epsilon ,\gamma }\|=1+\epsilon$ and $\nu (S_{\epsilon ,\gamma},x)=\max\{\nu(T,x)-\gamma,0\}$ at each $x\in\mathbb P^n$. \\
Let $A=E^+_{2/5}(T)$. Then $\nu (S_{\epsilon ,\gamma },x)>2/5-\gamma$ for $x\in A$. Since $S_{\epsilon ,\gamma }$ is smooth outside a finite set the measure $S_{\epsilon ,\gamma}\wedge T$ is well defined \cite{D93}. We estimate $|A|$ as follows: 
\begin{eqnarray*}
1+\epsilon=\int_{\mathbb P^n}S_{ \epsilon,\gamma} \wedge T\geq \sum _{x\in A}\nu ( S_{\epsilon,\gamma},x)\nu (T,x)>\frac{2}{5}\left(\frac{2}{5}-\gamma\right)|A|.
\end{eqnarray*}
Choosing $\epsilon>0$ very small and $\gamma >1/3$ very close to $1/3$ we find that $|A|\leq37$. 
\end{proof}

\bigskip

\par The argument in Case 2 of the proof of Theorem \ref{T:mt31} can be used to prove a more general result. 

\begin{Theorem}\label{T:KC} Let $(X,\omega)$ be a compact K\"ahler manifold of dimension $n$, and $T$ be a positive closed current of bidimension $(p,p)$ on $X$. For any $c>0$, the set $E_c(T)$ is contained in an analytic set of dimension $\leq p$ whose volume and number of irreducible components are bounded above by a constant $K(\|T\|,c)$ depending only on $\|T\|$ and $c$. 
\end{Theorem}

\begin{proof} Recall that $\|T\|=\int_XT\wedge\omega^p$ and if $Z$ is an analytic subvariety of $X$ then ${\rm vol}\,Z=\sum_V\int_V\omega^{\dim V}$, where the sum is over all irreducible components $V$ of $Z$. By Lelong's theorem, there is a positive number $\mu _0$ such that any subvariety of $X$ has volume at least $\mu _0$. Therefore the number of irreducible components of $Z$ is $\leq({\rm vol}\,Z)/\mu_0$.

\par The proof is by induction on $p$. If $p=0$ then $T$ is a measure and $E_c(T)$ is a finite set whose cardinality is $\leq\|T\|/c$.

\par Assume that the theorem is true for $p=p_0$. We need to prove it for $p=p_0+1$. Let us define $A_{c/2,p_0+1}(T)$ to be the union of all irreducible components of dimension $p_0+1$ of the analytic set $E_{c/2}(T)$. Set 
\begin{eqnarray*}
T'=T-\sum_{V\subset A_{c/2,p_0+1}(T)}\lambda _V(T)[V]\,,
\end{eqnarray*}
where the sum is over all irreducible components $V$ of $A_{c/2,p_0+1}(T)$ and $\lambda _V(T)$ is the generic Lelong number of $T$ along $V$. By \cite{Siu74} $T'$ is a positive closed current of bidimension $(p_0+1,p_0+1)$ and $\|T'\|\leq\|T\|$. Moreover the set $E_{c/2}(T')$ has dimension at most $p_0$, since $E_{c/2}(T')\subset E_{c/2}(T)$ and $T'$ does not charge any irreducible component $V$ of $A_{c/2,p_0+1}(T)$.  Since $\lambda_V(T)\geq c/2$, we have that
\begin{eqnarray*}
{\rm vol}\,A_{c/2,p_0+1}(T)=\big\|\sum_{V\subset A_{c/2,p_0+1}(T)}[V]\big\|\leq 2\|T\|/c .
\end{eqnarray*}

By \cite[Theorem 3.1]{Vig} there is a positive closed current $R$ of bidegree $(1,1)$ on $X$ which has the same Lelong number as $T'$ at every point and such that $\|R\|\leq C_1\|T'\|$, where $C_1>0$ is a constant depending only on $X$ and $\omega$. By Demailly's regularization theorem applied to $R$ (Main Theorem 1.1 in \cite{D92}), there is a positive closed current $R'$ of bidegree $(1,1)$ on $X$ such that: $\|R'\|\leq C_2\|R\|$, where $C_2$ is a constant depending only on $X$ and $\omega$, $R'$ is smooth on $X\setminus E_{c/2}(T')$, and $\nu (R',x)=\max \{\nu (T',x)-c/2 ,0\}$ for every $x\in X$. Since $\dim E_{c/2}(T')\leq p_0$, $T_1=R'\wedge T'$ is a well defined positive closed current of bidimension $(p_0,p_0)$ by \cite{D93}. Moreover, $\|T_1\|\leq C_3\|T'\|\|R'\|\leq C\|T\|^2$, where $C=C_1C_2C_3$ and $C_3$ is a constant depending only on $X$ and $\omega$. By Demailly's comparison theorem for Lelong numbers \cite{D93} we have for $x\in E_c(T)\setminus A_{c/2,p_0+1}(T)$,
\begin{eqnarray*}
\nu (T_1,x)\geq \nu (R',x)\nu (T',x)\geq c^2/2.
\end{eqnarray*}
Therefore, if $W$ is the union of all the irreducible components of $E_c(T)$ that are not contained in $A_{c/2,p_0+1}(T)$, then $W\subset E_{c^2/2}(T_1)$. The  induction assumption implies that $E_{c^2/2}(T_1)$ is contained in an analytic subset of dimension $\leq p_0$ whose volume is  $\leq K(C\|T\|^2,c^2/2 )$. Thus the proof for the case $p=p_0+1$ is complete.
 \end{proof}

\par We note that it is not true that the number of irreducible components of the set $E_c(T)$ itself is bounded by a constant depending only on $\|T\|$ and $c$, as the following simple example shows. Let $L_j$, $0\leq j\leq k+1$, be lines in $\mathbb P^n$ so that $L_0\cap L_j=\{z_j\}$, $j\geq 1$, and no three of them pass through the same point. Let 
$$T=\left(\frac{1}{2}-\frac{1}{2k}\right)[L_0]+\frac{1}{2k}\sum_{j=1}^{k+1}[L_j]\,.$$
Then $\|T\|=1$ and $E_{1/2}(T,\mathbb P^n)=\{z_1,\dots,z_{k+1}\}$, provided that $k\geq3$.

\section{Positive closed currents on $\mathbb P^m\times\mathbb P^n$}\label{S:multiproj}

\par We prove here certain geometric properties of the upper level sets of Lelong numbers of positive closed currents of bidimension $(1,1)$ on a multiprojective space 
$$X={\mathbb P}^m\times{\mathbb P}^n={\mathbb P}^m_z\times{\mathbb P}^n_w\,.$$ 
Let $\pi_z:X\longrightarrow{\mathbb P}^m_z$, $\pi_w:X\longrightarrow{\mathbb P}^n_w$, denote the canonical projections and 
$$z=[z_0:\ldots:z_m]\,,\;w=[w_0:\ldots:w_n]\,,$$ 
denote the homogeneous coordinates on $\mathbb P^m$, respectively on $\mathbb P^n$. Set 
$$\omega_z=\pi_z^\star\omega_m\,,\;\omega_w=\pi_w^\star\omega_n\,,$$
where $\omega_m$ and $\omega_n$ are the Fubini-Study forms on ${\mathbb P}^m$, respectively $\mathbb P^n$. The Dolbeault cohomology group $H^{m+n-1,m+n-1}(X,\mathbb R)$ is generated by the forms $\omega_z^m\wedge\omega_w^{n-1}$ and $\omega_z^{m-1}\wedge\omega_w^n$. Let 
$$\theta_{a,b}=a\omega_z^m\wedge\omega_w^{n-1}+b\omega_z^{m-1}\wedge\omega_w^n\,,\;a,b\geq0\,,$$
and let $\mathcal T_{a,b}$ denote the space of positive closed currents of bidimension $(1,1)$ on $X$ which lie in the cohomology class of $\theta_{a,b}$. 

\begin{Proposition}\label{P:proj1}
If $T\in\mathcal T_{a,b}\,$ then $E^+_{(a+b)/2}(T,X)\subset\pi_z^{-1}(x)$ for some $x\in\mathbb P^m$ or  $E^+_{(a+b)/2}(T,X)\subset\pi_w^{-1}(y)$ for some $y\in\mathbb P^n$.
\end{Proposition}

\begin{proof} We may assume that $a+b=1$ and that for any $x\in\mathbb P^m$ we have $E^+_{1/2}(T,X)\not\subset\pi_z^{-1}(x)$. Then there exist points $p_1,p_2\in E^+_{1/2}(T,X)$ such that $\pi_z(p_1)\neq\pi_z(p_2)$. 

\par We claim that $\pi_w(p_1)=\pi_w(p_2)$. Indeed, suppose $\pi_w(p_1)\neq\pi_w(p_2)$. Composing with an automorphism of $X$ we may assume that 
$$p_1=([1:0:\ldots:0],[1:0:\ldots:0])\,,\;p_2=([0:\ldots:0:1],[0:\ldots:0:1])\,.$$
Consider the function on $\mathbb C^{m+1}\times\mathbb C^{n+1}$, 
$$u(z_0,\dots,z_m,w_0,\dots,w_n)=\max\left\{\big(\sum_{j=1}^m|z_j|^2\big)\big(\sum_{k=0}^{n-1}|w_k|^2\big),\big(\sum_{j=0}^{m-1}|z_j|^2\big)\big(\sum_{k=1}^n|w_k|^2\big)\right\}.$$
The current $\frac{1}{2}\,dd^c\log u$ determines a positive closed current $S$ of bidegree $(1,1)$ on $X$ in the cohomology class of $\omega_z+\omega_w$ (see e.g. \cite{G02,CG09}). Note that $S$ has bounded local plurisubharmonic potentials on $X\setminus\{p_1,p_2\}$ and $\nu(S,p_1)=\nu(S,p_2)=1$. Then 
$$a+b=\int_XT\wedge S\geq T\wedge S(\{p_1\})+T\wedge S(\{p_2\})\geq\nu(T,p_1)+\nu(T,p_2)>1\,,$$
a contradiction.

\par Hence  $\pi_w(p_1)=\pi_w(p_2)=y$. We show that $E^+_{1/2}(T,X)\subset\pi_w^{-1}(y)$. If not, there exists $p\in E^+_{1/2}(T,X)$ with $\pi_w(p)\neq y$. Since $\pi_z(p_1)\neq\pi_z(p_2)$ we may assume that $\pi_z(p)\neq\pi_z(p_1)$. Then we obtain a contradiction as above, working with the points $p,p_1$ instead of the points $p_1,p_2$. 
\end{proof}

\par Our next result is in analogy to that of Theorem \ref{T:mt1}. By {\em vertical line in $X$} we mean a line $L\subset\pi_z^{-1}(x)\equiv\mathbb P^n$ for some $x\in\mathbb P^m$, while by {\em horizontal line in $X$} we mean a line $L\subset\pi_w^{-1}(y)\equiv\mathbb P^m$ for some $y\in\mathbb P^n$.

\begin{Proposition}\label{P:proj2}
Let $T\in\mathcal T_{a,b}\,$ and $\alpha=\max\left\{\frac{2a+b}{3}\,,\frac{a+2b}{3}\right\}$.
Then $E^+_\alpha(T,X)$ is contained in a vertical line in $X$ or in a horizontal line in $X$. 
\end{Proposition}

\begin{proof} Since $\alpha\geq(a+b)/2$ it follows by Proposition \ref{P:proj1} that $E^+_\alpha(T,X)\subset\pi_z^{-1}(x)$ for some $x\in\mathbb P^m$ or  $E^+_\alpha(T,X)\subset\pi_w^{-1}(y)$ for some $y\in\mathbb P^n$. Without loss of generality we may assume that $E^+_\alpha(T,X)\subset\pi_z^{-1}(x)$, where $x=[1:0:\ldots:0]\in\mathbb P^m$. We will show that $E^+_\alpha(T,X)$ is contained in a line in $\pi_z^{-1}(x)\equiv\mathbb P^n$.

\par Suppose for a contradiction that there exist non-collinear points $y^1,y^2,y^3\in\mathbb P^n$ such that $p_j:=(x,y^j)\in E^+_\alpha(T,X)$. We can find homogeneous quadratic polynomials $P_1,\dots,P_n$ on $\mathbb C^{n+1}$ such that the set $\{P_1=0\}\cap\ldots\cap\{P_n=0\}\subset\mathbb P^n$ is finite and it contains the points $y^1,y^2,y^3$. Moreover, the $2\omega_n$-plurisubharmonic function on $\mathbb P^n$ determined by $\frac{1}{2}\log\big(\sum_{k=1}^n|P_k|^2\big)$ has Lelong number 1 at each $y^j$. Consider the function on $\mathbb C^{m+1}\times\mathbb C^{n+1}$, 
$$u(z,w)=\max\left\{\big(\sum_{j=0}^m|z_j|^2\big)\big(\sum_{k=1}^n|P_k(w)|^2\big),\big(\sum_{j=1}^m|z_j|^2\big)\big(\sum_{k=0}^n|w_k|^2\big)^2\right\}.$$

The current $\frac{1}{2}\,dd^c\log u$ determines a positive closed current $S$ of bidegree $(1,1)$ on $X$ in the cohomology class of $\omega_z+2\omega_w$ (see e.g. \cite{G02,CG09}). Note that $S$ has bounded local plurisubharmonic potentials on the complement of a finite subset of $X$ and $\nu(S,p_j)=1$. Then 
$$2a+b=\int_XT\wedge S\geq\sum_{j=1}^3T\wedge S(\{p_j\})\geq\sum_{j=1}^3\nu(T,p_j)>3\alpha\,,$$
a contradiction. This completes the proof.
\end{proof}

\medskip

\par We end this section with some examples showing that Propositions \ref{P:proj1} and \ref{P:proj2} are sharp. Consider distinct points $x_1,x_2\in\mathbb P^m$, $y_1,y_2\in\mathbb P^n$, and let $p_{jk}=(x_j,y_k)\in X$. For $j=1,2$, denote by $V_j\subset\pi_z^{-1}(x_j)$ the vertical line determined by the points $p_{j1},p_{j2}$, and by $H_j\subset\pi_w^{-1}(y_j)$ the horizontal line determined by the points $p_{1j},p_{2j}$. Let $a,b>0$ and $T_1,T_2\in\mathcal T_{a,b}$ be the currents 
$$T_1=\frac{a}{2}\,([V_1]+[V_2])+b[H_1]\,,\;T_2=a[V_1]+\frac{b}{2}\,([H_1]+[H_2])\,.$$
Then 
\begin{eqnarray*}
&&\{p_{11},p_{21}\}\subset E^+_{(a+b)/2}(T_1,X)\subset H_1\subset\pi_w^{-1}(y_1)\,,\\
&&\{p_{11},p_{12}\}\subset E^+_{(a+b)/2}(T_2,X)\subset V_1\subset\pi_z^{-1}(x_1)\,.
\end{eqnarray*}
Hence the set $E^+_{(a+b)/2}(T,X)$ can be contained in a vertical fiber or in a horizontal fiber, regardless of how $a$ compares to $b$. Assume next that $a\geq b$ and let 
$$T_3=\frac{a+b}{2}\,[V_1]+\frac{a-b}{2}\,[V_2]+b[H_1]\,,\,\text{ so }\,V_1\cup\{p_{21}\}\subset E_{(a+b)/2}(T_3,X)\,.$$
Hence $E_{(a+b)/2}(T_3,X)\not\subset\pi_z^{-1}(x)$ for any $x\in\mathbb P^m$, $E_{(a+b)/2}(T_3,X)\not\subset\pi_w^{-1}(y)$ for any $y\in\mathbb P^n$, and Proposition \ref{P:proj1} is sharp.

\par The next example shows the sharpness of Proposition \ref{P:proj2}. Assume $a\geq b$, let $x\in\mathbb P^m$, let $y_1,y_2,y_3\in\mathbb P^n$ be non-collinear points, and set $p_j=(x,y_j)$. Denote by $V_{jk}\subset\pi_z^{-1}(x)$ the vertical line determined by the points $p_j,p_k$, and by $H_j\in\pi_w^{-1}(y_j)$ a horizontal line containing $p_j$. Then 
$$T_4=\frac{a}{3}\,([V_{12}]+[V_{23}]+[V_{13}])+\frac{b}{3}\,([H_1]+[H_2]+[H_3])\in\mathcal T_{a,b}$$
has $\nu(T_4,p_j)=(2a+b)/3=\alpha$, where $\alpha$ is as in Proposition \ref{P:proj2}. Thus $E_\alpha(T_4,X)$ is not contained in a vertical line or in a horizontal line in $X$.


\begin{thebibliography}{XXXXX}



\bibitem{C06} D. Coman, {\em Entire pluricomplex Green functions and Lelong numbers of projective currents}, Proc. Amer. Math. Soc. {\bf 134} (2006), 1927--1935.

\bibitem{CG09} D. Coman and V. Guedj, {\em Quasiplurisubharmonic Green functions}, J. Math. Pures Appl. (9) {\bf 92} (2009), 456--475.


\bibitem{D92} J. P. Demailly, {\em Regularization of closed positive currents and intersection theory}, J. Algebraic Geom. {\bf 1} (1992), 361--409.

\bibitem{D93} J. P. Demailly, {\em Monge-Amp\`ere operators, Lelong numbers and intersection theory}, in {\em Complex analysis and geometry}, Plenum, New York, 1993, 115--193.

\bibitem{EH87} D. Eisenbud and J. Harris, {\em On varieties of minimal degree (a centennial account)}, Algebraic geometry, Bowdoin, 1985 (Brunswick, Maine, 1985), 3--13, Proc. Sympos. Pure Math., {\bf 46}, Part 1, Amer. Math. Soc., Providence, RI, 1987.

\bibitem{FS95} J. E. Forn\ae ss and N. Sibony, {\em Oka's inequality for currents and  applications}, Math. Ann. {\bf 301} (1995), 399--419.

\bibitem{GH}  Ph. Griffiths and J. Harris, {\em Principles of Algebraic Geometry}, Pure and Applied Mathematics, Wiley-Interscience, New York, 1978. 

\bibitem{G02} V. Guedj, {\em Dynamics of polynomial mappings of ${\mathbb C}^2$}, Amer. J. Math. {\bf 124} (2002), 75--106.

\bibitem{GZ05} V. Guedj and A. Zeriahi, {\em Intrinsic capacities on compact K\"ahler manifolds}, J. Geom. Anal. {\bf 15} (2005), 607--639.


\bibitem{Me98} M. Meo, {\em In\'egalit\'es d'auto-intersection pour les courants positifs ferm\'es d\'efinis dans les vari\'et\'es projectives},  Ann. Scuola Norm. Sup. Pisa Cl. Sci. (4) {\bf 26} (1998), 161--184.


\bibitem{Siu74} Y. T. Siu, {\em Analyticity of sets associated to Lelong numbers and the extension of closed positive currents}, Invent. Math. {\bf 27} (1974), 53--156.

\bibitem{Vig} G. Vigny, {\em Lelong-Skoda transform for compact K\"ahler manifolds and self-intersection inequalities}, J. Geom. Anal.  {\bf 19} (2009), 433--451.

\end{thebibliography}
\end{document}